\documentclass{article}
\usepackage{amsthm,amsmath,amssymb}
\usepackage{authblk}
\newtheorem{thm}{Theorem}[section]
\theoremstyle{definition}
\newtheorem{dfn}{Definition}[section]
\theoremstyle{remark}

\theoremstyle{plain}
\newtheorem{lem}[thm]{Lemma}
\newtheorem{cnj}[thm]{Conjecture}
\usepackage[toc,page]{appendix}
\title{Improving the bound for maximum degree on Murty-Simon Conjecture}

\author[1]{A. Bahjati\footnote{amin.bahati@gmail.com}}
\author[1]{A. Jabalameli \footnote{ajabalameli@ce.sharif.edu}}
\author[1]{M. Ferdosi\footnote{mohsenferdosi@gmail.com}}
\author[1]{M. M. Shokri\footnote{M.MahdiShokri@gmail.com}}
\author[2]{M. Saghafian\footnote{ms.saghafian@gmail.com}}
\author[2]{S. Bahariyan\footnote{sorush.bahariyan@gmail.com}}

\affil[1]{Department of Computer Engineering, Sharif University of Technology, Tehran, Iran}
\affil[2]{Department of Mathematical Sciences, Sharif University of Technology, Tehran, Iran}

\begin{document}
\maketitle
\begin{abstract}
	A graph is said to be diameter-$k$-critical if its diameter is $k$ and removal of any of its edges increases its diameter.  A beautiful conjecture by Murty and Simon, says that every diameter-2-critical graph of order $n$ has at most $\lfloor n^2/4\rfloor$ edges and equality holds only for $K_{\lceil n/2 \rceil,\lfloor n/2 \rfloor }$. Haynes et al. proved that the conjecture is true for $\Delta\geq 0.7n$. They also proved that for $n>2000$, if $\Delta \geq 0.6789n$ then the conjecture is true. We will improve this bound by showing that the conjecture is true for every $n$ if $\Delta\geq\ 0.6755n$.
\end{abstract}
\section{Introduction}
Throughout this paper we assume that $G$ is a simple graph. Our notation is the same as \cite{survey}, let $G=(V,E)$ be a graph with vertex set $V$ of order $n$ and edge set $E$ of size $m$. For a vertex $v \in G$ we denote the set of its neighbors in $G$ by $N_G(u)$. Also we denote $N_G(u) \cup u$ by $N_G(u)$. The maximum and minimum degrees of $G$ will be denoted by $\Delta$ and $\delta$, respectively. The distance $d_G(u,v)$ between two vertices $u$ and $v$ of $G$, is the length of the shortest path between them. The \emph{diameter} of $G$, ($diam(G)$), is the maximum distance among all pairs of vertices in $G$.

We say graph $G$ is \emph{diameter-$k$-critical} if its diameter is $k$ and removal of any of its edges increases its diameter.
Based on a conjecture proposed by Murty and Simon \cite{conj}, there is an upper bound on the number of edges in a diameter-2-critical graph.

\begin{cnj}
	Let $G$ be a diameter-2-critical graph. Then $m \leq {[n^2/4]}$ and equality holds only if $K_{\lceil n/2 \rceil,\lfloor n/2 \rfloor }$.
\end{cnj}

Several authors have conducted some studies on the conjecture proving acceptable results nearly close to the original one, however, no complete proof  has been provided yet.
Plesnık \cite{plensik} showed that $m < \frac{3n(n - 1)}{8}$. Moreover, Caccetta and Haggkvist \cite{conj} proved
$m < 0.27n^2$. Fan \cite{fan} also proved the fact that for $n\leq24$ and for $n=26$ we have $m \leq {[\frac{n^2}{4}]}$. For
$n=25$, he achieved $m < \frac{n^2}{4} + \frac{(n^2 - 16.2n + 56)}{320} < 0.2532n^2$.  Another proof was presented by Xu \cite{Xu} in 1984, which was found out to have a small error. Afterwards, Furedi \cite{furedi} provided a considerable result showing that 
the original conjecture is true for large $n$, that is, for $n > n_0$ where $n_0$ is a tower of 2s of height about $10^{14}$.
This result is highly significant though not applicable to those graphs we are currently working with. 
\section{Total Domination}
	Domination number and Total domination number are parameters of graphs which are studied, respectively, in \cite{Haynes98a,advanced_domination_book} and \cite{total_domination_book}.  
	Assume $G=(V,E)$ is a simple graph. Let $X$ and $Y$ be subsets of $V$; We say that $X$ dominates $Y$, written $X \succ Y$,  if and only if every element of $Y-X$ has a neighbor in $X$.
	Similarly, we say that $X$ totally dominates $Y$, written $X \succ_t Y$  if and only if every element of $Y$ has a neighbor in $X$.
	If $X$ dominates or totally dominates $V$, we might write,  $X \succ G$ or  $X \succ_t G$ instead of $X \succ V$ and $X \succ_t V$,  respectively.
	Domination number and total domination number of $G=(V,E)$ are the size of smallest subset of $V$ that, dominates and totally dominates $V$, respectively.
	A graph $G$ with total domination number of $k$ is called $k_t$-critical, if every graph constructed by adding an edge between any nonadjacent vertices of $G$ has total domination number less than $k$.
	It is obvious that adding any edge to $k_t$-critical graph $G$ would result a graph which has total domination number of $k-1$ or $k-2$.   
	Assume $G$ is $k_t$-critical graph. If for every pair of non adjacent vertices $\{u,v\}$ of $G$, the total domination number of $G+uv$ is $k-2$, then $G$ is called $k_t$-supercritical.
	As shown in  \cite{hanson03} there is a great connection between diameter-2-critical graphs and total domination critical graphs:
	\begin{thm} \emph{(\cite{hanson03})}\label{complement}
		A graph is diameter-2-critical if and only if its complement is $3_t$-critical or $4_t$-supercritical. 
	\end{thm}
	By this theorem in order to prove Murty-Simon conjecture, it suffices to prove that every graph which is $3_t$-critical, or $4_t$-critical , has at least $\lfloor n(n-2)/4 \rfloor$ edges where $n$ is order of graph. This problem is solved in some cases in \cite{haynes98b,haynes98c,haynes11} :
	\begin{thm} \emph{(\cite{haynes98b})}\label{4t}
		A graph $G$ is $4_t$-supercritical if and only if G is disjoint union of two nontrivial complete graphs.
	\end{thm}
	\begin{thm} \emph{(\cite{haynes98c})}\label{diam23}
		If $G$ is a $3_t$-critical graph, then $2 \le diam(G) \le 3$.
	\end{thm}
	\begin{thm} \emph{(\cite{haynes11})}
			Every $3_t$-critical graph of diameter 3 and order n has size $m \ge  n(n-2)/4  $.
	\end{thm}
	By this theorems a proof for following conjecture will show that Murty-Simon conjecture is true.
	\begin{cnj}
		A $3_t$-critical graph of order $n$ and of diameter 2 has size $m \ge n(n-2)/4$.
	\end{cnj}
	More recently Haynes et al. proved the following: 
	\begin{thm}\label{this} \emph{(\cite{haynes14})}
		Let $G$ be a $3_t$-critical graph of order $n$ and size $m$. Let $\delta = \delta (G)$. Then the following holds:\\
		a)If $\delta \geq 0.3n $, then $ m > \lceil n(n-2)/4 \rceil $.\\
		b)If $n \ge 2000$ and $\delta \geq 0.321n$, then $ m > \lceil n(n-2)/4 \rceil $.
	\end{thm}
	Also G. Fan et al. proved that:
	\begin{thm}\label{n<25} \emph{(\cite{fan})}
			The Murty-Simon conjecture is true for every graph with less than 25 verices.  
		\end{thm}
	In next section, in order to improve this bound, we will prove that, every simple diameter-2-critical graph of order $n$ and size $m$ satisfies $ m <\lfloor n^2/4 \rfloor $ if  $\Delta\geq\ 0.6756n$.
	
\section{Main Result}
	In this section we will prove Murty-Simon conjecture for graphs which their complement are $3_t$-critical and have less restriction on their minimum degree and improve the result proposed by Haynes et al in \cite{haynes14}. First we recall the following lemma, which was proposed in that paper.
\begin{lem}
	Let $u$ and $v$ are nonadjacent vertices in $3_t$-critical graph $G$, clearly  $\{u,v\} \nsucc G$.
	Then there exists a vertex $w$, such that $w$ is adjacent to exactly one of $u,v$, say $u$, and $\{u,w\} \succ G-v$. We will call $uw$ \emph{quasi-edge} associated with $uv$. Further $v$ is the unique vertex not dominated by $\{u,w\}$ in $G$; In this case we call $v$ \emph{supplement} of $\{u,w\}$.
\end{lem}
\begin{dfn}  
	Let $G=(V,E)$ be a $3_t$-critical graph.
	If $S\subseteq V$ then we say that  $S$ is a \emph{quasi-clique} if for each nonadjacent pair of vertices of $S$ there exists a quasi-edge associated with that pair, and  each quasi-edge associated with that pair at contains at least on vertex outside $S$.
	Edges \emph{associated with} quasi-clique $S$ are the union of the edges with both ends in $S$ and the quasi-edges associated with some pair of nonadjacent vertices of $S$. 
\end{dfn}
\begin{dfn}  
	Let $G=(V,E)$ be a $3_t$-critical graph.
	Let $A$ and $B$ be two disjoint subsets of $V$. We define $E(G;A,B)$ as set of all edges $\{a,b\}$ where $a \in A$ and $b \in B$, and $\{a,b\}$ is associated with a non adjacent pair $\{a,c\}$, where $c$ is in $A$. By lemma 3.1, we know that every two members of $E(G;A,B)$ are associated with different non adjacent pairs. 
	
\end{dfn}
\begin{lem}\label{shimi}
	Let $G$ be a $3_t$-critical graph.
	Let $S \subset V(G)$, if $S^*=\cap_{s\in S} N(s)$ ,then the following holds: $$|E(G[S^*])|+|E(G;S^*,V(G)-(S^*\cup S))|\geq  \frac{|S^*|^2-2|S^*|}{c}$$
	
	Where $c$ is the greatest root of $x^2-4x-4=0$, which is equal to $2+2\sqrt2\approx4.83$.
\end{lem}
\begin{proof}
		We apply induction on size of $S^*$ to prove the theorem.
		Note that for every pair of non-adjacent vertices in $S^*$ such as $\{u,v\}$, If $\{u,w\}$ is the quasi-edge associated to it, then, since $v$ is adjacent $u$, we can conclude that $w\not\in S$.
		Note that when $|S^*| \le 2$, since $\frac{|S^*|^2-2|S^*|}{c} \leq 0$, then the inequality is obviously true. 
		Let $v$ be the vertex having minimum degree in $G[S^*]$. We denote the set of neighbors of $v$ in $S^*$ by $A$.
		Since every vertex in $S^* - (A \cup \{v\})$ is not adjacent to $v$, so $S^* - (A \cup \{v\})$ is a quasi-clique.
		Also $A$ is $\cap_{s\in {S \cup \{v\}}} N(s)$, so $|E(G[A])|+|E(G;A,V(G)-(A\cup S \cup \{v\}))| \geq \frac{|A|^2-2|A|}{c}$.
		For every pair of non-adjacent vertices $\{x,y\}$, one of them is the supplement of quasi-edge associated to this pair, so quasi-edges associated to non-adjacent pairs in $A$ and $S^*-(A\cup\{v\})$ are disjoint. With statements mentioned above we can conclude that:\\ $$|E(G[S^*])|+|E(G;S^*,V(G)-(S^*\cup S))| \geq \frac{|A|^2-2|A|}{c} + {{|S^*| - |A| - 1} \choose 2} + |A|.$$\\
		The right side of the inequality is a function of $|A|$, that we call it $f(|A|)$. One can find out that:\\ $$f'(|A|)=\frac{(c+2)|A|}{c}+(\frac{5}{2}-\frac{2}{c})-|S^*|$$\\ So $f'(|A|)$ has negative value whenever $0 \le |A| \le \frac {2|S^*|-4} {c}$ and $|S^*| \geq 3$. So it suffices to prove that $f(\frac {2|S^*|-4} {c})\ge \frac{|S^*|^2-2|S^*|}{c}$, which is done by Lemma \ref{appendProoflem}.
		On the other hand when $|A| \geq \frac {2|S^*|-4} {c}$ by definition of $A$, we can easily conclude that: $$|E(G[S^*])|\geq\frac{|A||S^*|}{2}\ge \frac{|S^*|^2-2|S^*|}{c}.$$
\end{proof}
\begin{lem}\label{B_qc}
	Let $G=(V,E)$ be a $3_t$-critical graph.
	If $v \in V$, then $V-N_{G}[v]$ is a quasi-clique.
\end{lem}
\begin{proof}
	This lemma is generalized from a lemma in \emph{(\cite{haynes14})}, in which $v$ was assumed as a vertex with minimum degree in $G$. Since the proof was independent of such assumption, the same proof is correct. 
\end{proof}

Now, we present the main result of this paper:
\begin{thm}\label{mainTh}
	Suppose that $c=2+\sqrt2$, and $a$ is the smallest root of the equation $(2c+4)x^2-4cx+c=0$, which is equal to $\frac{\sqrt2-\sqrt{2-\sqrt2}}{2}\approx0.32442$. Let $G(V,E)$ be a $3_t$-critical graph of order $n$, size $m$ and minimum degree $\delta$.
	If $n\ge3$ and $\delta \le an-1$ then,
	$$m>\lceil{\frac{n(n-2)}{4}}\rceil$$
\end{thm}
\begin{proof}
	First, note that for every positive integer $n$:
	\begin{itemize}
		\item if $n$ is even ${n(n-2)}$ is divisible by $4$.
		\item if $n$ is odd ${n(n-2)+1}$ is divisible by $4$.
	\end{itemize}
	So it suffices to prove that:
	$$m>\frac{n(n-2)+1}{4}$$
		Let $v \in V(G)$ be a vertex with $\delta$ neighbors and $A=N_G(v)$. َAlso let $B=V-N_G[v]$, then by Lemma \ref{B_qc}, $B$ is a quasi-clique. Also by Lemma \ref{shimi}, $|E(G[A])|+|E(G;A,B)|\ge \frac{\delta^2-2\delta}{c}.$ 
		$A$ and $B$ are disjoint, so the quasi-edges associated to non-adjacent pairs in $A$ are disjoint from the quasi-edges associated to non-adjacent pairs in $B$, because every quasi-edge has unique supplement. Therefore, we have:
		$$m\ge\delta+\frac{\delta^2-2\delta}{c}+{n-1-\delta\choose2}$$So by Lemma \ref{appendProofTheorem} we have: $$m>\frac{n(n-2)+1}{4}$$\\
		
	\end{proof}
	\begin{thm}
		For every diameter-2-critical graph $G$ of order $n$ and size $m$, if $\Delta(G)\ge0.6756n$, then $m<\lfloor{\frac{n^2}{4}}\rfloor$
	\end{thm}
	\begin{proof}
		Since $diam(G)=2$, so $n\ge3.$ 
		Let $\bar{G}$ be complement of $G$. Assume that size of $\bar{G}$ is $m'$. Since $m+m'={n\choose2}$, so it suffices to prove that: 
		$$m'>\lceil\frac{n(n-2)}{4}\rceil.$$ 
		We have: $$\delta(\bar{G})=n-1-\Delta(G)\le0.3244n-1$$		
		Note that by Theorem \ref{complement}, $\bar{G}$ is either $3_t$-critical or $4_t$-supercritical. If $\bar{G}$ is $4_t$-supercritical, then by Theorem \ref{4t}, $\bar{G}$ is disjoint union of two non-trivial graph and size of the smaller one is less than $0.3244n-1$, which means $$m'\ge{0.3244n-1\choose2}+{0.6756n+1\choose2}>\lceil\frac{n(n-2)}{4}\rceil.$$ So we may consider that $\bar{G}$ is $3_t$-critical, which is shown in Theorem \ref{mainTh}.
		\end{proof}

	\appendix
	\section{\\Proof of Inequalities} \label{App:AppendixA}
\begin{lem}\label{appendProofTheorem}
	Suppose that $c=2+\sqrt2$, and $a$ is smaller root of the equation $(2c+4)x^2-4cx+c=0$, which is eqaul to $\frac{\sqrt2-\sqrt{2-\sqrt2}}{2}\approx0.3244$. \\ If  $an-1\ge y\ge0$ and $n\ge3$ then:\\ 
	$$y+{\frac{y^2-2y}{c}} +   {n-1-y \choose 2} > \frac{n(n-2)+1}{4} $$
		
\end{lem}
\begin{proof}
		Let $f(y)= y+{\frac{y^2-2y}{c}} +   {n-1-y \choose 2}  $. We have:$$f'(y)=1+\frac{2y-2}{c}-n+y+\frac{3}{2}$$
		$$=-n+\frac{5}{2}+\sqrt2y-\frac{2}{c}$$
		$$<-n+\frac{5}{2}+\sqrt2(an-1)-\frac{2}{c}<0
		$$
		Which means $f(y)\ge f(an-1)$. Let $g(n)=f(an-1)-\frac{n(n-2)+1}{4}$. Now it suffices to prove $g(n)$ has positive value for every $n\ge3$. $$g(n)=\frac{1}{4}((-8+7\sqrt2+4\sqrt{4-2\sqrt2}-7\sqrt{2-2\sqrt2})n+6\sqrt2-11)$$
		So the coefficient of $n$ is positive and $g(3)\approx0.025>0$, so we can conclude that $g(n)$ is positive when $n\ge3$. 
\end{proof}

\begin{lem}\label{appendProoflem}
	Let $n\ge3$  be a positive integer and $c=2+2\sqrt2$, then $$\frac{(\frac{2n-4}{c})^2-2(\frac{2n-4}{c})}{c}+{n-(\frac{2n-4}{c})-1\choose2}+(\frac{2n-4}{c})\ge\frac{n^2-2n}{c}.$$
\end{lem}
\begin{proof}
	We prove that $f(n)=\frac{(\frac{2n-4}{c})^2-2(\frac{2n-4}{c})}{c}+{n-(\frac{2n-4}{c})-1\choose2}+(\frac{2n-4}{c})-\frac{n^2-2n}{c}$ \\\\
	has positive value. $$f(n)=\frac{1}{2}((3\sqrt{2}-4)n+(8-6\sqrt{2})$$$$=\frac{(3\sqrt{2}-4)}{2}(n-2)>0$$
	 So $f(n)$ is positive for $n\ge3$.  
\end{proof}
	\appendix
\end{document}